\documentclass[11pt,reqno]{amsart}

\usepackage{xspace}
\usepackage{amsmath,amsthm,amsfonts,amssymb,latexsym,mathrsfs,color}
\usepackage{hyperref}

\newtheorem{theorem}{Theorem}%[section]
\newtheorem{cor}[theorem]{Corollary}
\newtheorem{prop}[theorem]{Proposition}

\newcommand{\pk}{{\rm pk\,}}
\newcommand{\des}{{\rm des\,}}
\newcommand{\altdes}{{\rm altdes\,}}

\newcommand{\msn}{\mathfrak{S}_n}

\newcommand{\msnn}{\mathfrak{S}_{n+1}}

\newcommand{\lrf}[1]{\lfloor #1\rfloor}

\title{Derivative polynomials and enumeration of permutations by their alternating descents}
\author[S.-M.~Ma]{Shi-Mei Ma}
\address{School of Mathematics and Statistics,
        Northeastern University at Qinhuangdao,
         Hebei 066004, P.R. China}
\email{shimeimapapers@gmail.com (S.-M. Ma)}
\author[Y.-N. Yeh]{Yeong-Nan Yeh}
\address{Institute of Mathematics,
        Academia Sinica, Taipei, Taiwan}
\email{mayeh@math.sinica.edu.tw (Y.-N. Yeh)}
\begin{document}
\begin{abstract}
In this paper we present an explicit formula for the number of permutations with a given number of alternating descents.
Moreover, we study the interlacing property of the real parts of the zeros of the generating polynomials of these numbers.
\end{abstract}


\maketitle

\section{Introduction}
%%%%%%%%%%%%%%%%%%%%%%%%%%%%%%%%%%%%%%%%%%%%%%%%%%%%%%%%%%%%%%%%%%%%%%%%%%%%%%%%%
%%%%%%%%%%%%%%%%%%%%%%%%%%%%%%%%%%%%%%%%%%%%%%%%%%%%%%%%%%%%%%%%%%%%%%%%%%%%%%%%%
%%%%%%%%%%%%%%%%%%%%%%%%%%%%%%%%%%%%%%%%%%%%%%%%%
%%%%%%%%%%%%%%%%%%%%%%%%%%%5%%%%
Let $\msn$ denote the symmetric group of all permutations of $[n]$, where $[n]=\{1,2,\ldots,n\}$.
For a permutation $\pi\in\msn$, we define a {\it descent} to be a position $i$ such that $\pi(i)>\pi(i+1)$. Denote by $\des(\pi)$ the number of descents of $\pi$.
The classical Eulerian polynomials $A_n(x)$ are defined by
\begin{equation*}
A_n(x)=\sum_{\pi\in\msn}x^{\des(\pi)}=\sum_{k=0}^{n-1}A(n,k)x^{k}.
\end{equation*}
As a variation of the descent statistic,
the number of {\it alternating descents} of a permutation $\pi\in\msn$ is defined by
$$\altdes(\pi)=|\{2i: \pi(2i)<\pi(2i+1)\}\cup \{2i+1: \pi(2i+1)>\pi(2i+2)\}|.$$
We say that $\pi$ has a {\it 3-descent} at index $i$ if $\pi(i)\pi(i+1)\pi(i+2)$ has one of the patterns: $132$, $213$, or $321$. Chebikin~\cite{Chebikin08} showed that
the alternating descent statistic of permutations in $\msn$ is equidistributed
with the 3-descent statistic of permutations in $\{\pi\in\msnn : \pi_1=1\}$.
Then the equations
$$\widehat{A}_n(x)=\sum_{\pi\in\msn}x^{\altdes(\pi)}=\sum_{k=0}^{n-1}\widehat{A}(n,k)x^k$$
define the {\it alternating Eulerian polynomials} $\widehat{A}_n(x)$ and the {\it alternating Eulerian numbers} $\widehat{A}(n,k)$.
The first few $\widehat{A}_n(x)$ are given as follows:
\begin{align*}
   \widehat{A}_1(x)& =1,\\
  \widehat{A}_2(x)&=1+x, \\
  \widehat{A}_3(x)&=2+2x+2x^2, \\
  \widehat{A}_4(x)& =5+7x+7x^2+5x^3,\\
  \widehat{A}_5(x)&=16+26x+36x^2+26x^3+16x^4.
\end{align*}

Chebikin~\cite{Chebikin08} proved that
\begin{equation}\label{EGF-Anx}
\sum_{n\geq 1}\widehat{A}_{n}(x)\frac{z^n}{n!}=\frac{\sec(1-x)z+\tan(1-x)z)-1}{1-x(\sec(1-x)z+\tan(1-x)z)},
\end{equation}
and the numbers $\widehat{A}(n,k)$ satisfy
the recurrence relation
\begin{equation*}\label{Recu-Ank}
\sum_{i=0}^n\sum_{j=0}^k\binom{n}{i}\widehat{A}(i,j+1)\widehat{A}(n-i,k-j+1)=(n+1-k)\widehat{A}(n,k+1)+(k+1)\widehat{A}(n,k+2).
\end{equation*}

In recent years, several authors pay attention to the alternating descent statistic and its associated permutation statistics.
The reader is referred to~\cite{Gessel14,Remmel12} for recent progress on this subject. For example, Gessel and Zhuang~\cite{Gessel14} defined an alternating run to be a maximal consecutive subsequence with no alternating descents.
The purpose of this paper is to present an explicit formula for
the numbers $\widehat{A}(n,k)$.
In Section~\ref{sec:Explicit}, we express the polynomials $\widehat{A}_n(x)$ in terms of the {\it derivative polynomials} $P_n(x)$
defined by Hoffman~\cite{Hoffman95}:
$$P_n(\tan \theta)=\frac{d^n}{d\theta^n}\tan\theta.$$

\section{An explicit formula}\label{sec:Explicit}
%%%%%%%%%%%%%%%%%%%%%%%%%%%%%%%%%%%%%%%%%%%%%%%%%%%%%%%%%%%%%%%%%%%%%%%%%%%%%%%%%
%%%%%%%%%%%%%%%%%%%%%%%%%%%%%%%%%%%%%%%%%%%%%%%%%%%%%%%%%%%%%%%%%%%%%%%%%%%%%%%%%
%%%%%%%%%%%%%%%%%%%%%%%%%%%%%%%%%%%%%%%%%%%%%%%%%
%%%%%%%%%%%%%%%%%%%%%%%%%%%5%%%%
Let $D$ denote the differential operator $d/{d\theta}$. Set $x=\tan\theta$. Then $D(x^n)=nx^{n-1}(1+x^2)$ for $n\geq 1$. Thus $D^n(x)$ is a polynomial in $x$. Let
$P_n(x)=D^n(x)$.
Then $P_0(x)=x$ and
\begin{equation}\label{Pnx-recu}
P_{n+1}(x)=(1+x^2)P_n'(x).
\end{equation}
Clearly, $\deg P_n(x)=n+1$.
By definition, we have
\begin{equation}\label{EGF-Pnx}
\tan(\theta+z)=\sum_{n\geq 0}P_n(x)\frac{z^n}{n!}=\frac{x+\tan z}{1-x \tan z},
\end{equation}
Let $P_n(x)=\sum_{k=0}^{n+1}p(n,k)x^k$. It is easy to verify that
\begin{equation*}
p(n,k)=(k+1)p(n-1,k+1)+(k-1)p(n-1,k-1).
\end{equation*}
The first few terms can be computed directly as follows:
\begin{align*}
  P_1(x)& =1+x^2, \\
  P_2(x)& =2x+2x^3, \\
  P_3(x)& =2+8x^2+6x^4,\\
  P_4(x)& =16x+40x^3+24x^5.
\end{align*}
Note that $P_n(-x)=(-1)^{n+1}P_n(x)$ and $x\| P_{2n}(x)$.
Thus we have the following expression:
\begin{equation*}\label{pnk}
P_n(x)=\sum_{k=0}^{\lfloor(n+1)/2\rfloor}p(n,n-2k+1)x^{n-2k+1}.
\end{equation*}

There is an explicit formula for the numbers $p(n,n-2k+1)$.
\begin{prop}[{\cite[Proposition 1]{Ma122}}]
For $n\geq 1$ and $0\leq k\leq {\lfloor(n+1)/2\rfloor}$, we have
\begin{equation*}
p(n,n-2k+1)=(-1)^{k}\sum_{i\geq 1}i!{n \brace i}(-2)^{n-i}\left[\binom{i}{n-2k}-\binom{i}{n-2k+1}\right].
\end{equation*}
\end{prop}

Now we present the first main result of this paper.
\begin{theorem}
For $n\geq 1$, we have
\begin{equation}\label{AnxPnx:01}
2^n(1+x^2)\widehat{A}_n(x)=(1-x)^{n+1}P_n\left(\frac{1+x}{1-x}\right).
\end{equation}
\end{theorem}
\begin{proof}
It follows from~\eqref{EGF-Pnx} that
\begin{align*}
  \sum_{n\geq 1} (1-x)^{n+1}P_n\left(\frac{1+x}{1-x}\right)\frac{z^n}{n!}&=(1-x)\sum_{n\geq 1}P_n\left(\frac{1+x}{1-x}\right)\frac{(z-xz)^n}{n!} \\
&=(1+x^2)\frac{2\tan(z-xz)}{1-x-(1+x)\tan(z-xz)}.
\end{align*}

Comparing with~\eqref{EGF-Anx}, it suffices to show the following
\begin{equation}\label{identity}
\frac{\sec(2z-2xz)+\tan(2z-2xz)-1}{1-x(\sec(2z-2xz)+\tan(2z-2xz))}=\frac{2\tan(z-xz)}{1-x-(1+x)\tan(z-xz)}.
\end{equation}

Set $t=\tan(z-xz)$. Using the tangent half-angle substitution,
we have $$\sec(2z-2xz)=\frac{1+t^2}{1-t^2},~\tan(2z-2xz)=\frac{2t}{1-t^2}.$$
Then the left hand side of~\eqref{identity} equals
$$\frac{2t(1+t)}{1-t^2-x(1+t)^2}=\frac{2t}{1-x-(1+x)t}.$$
This completes the proof.
\end{proof}

It follows from~\eqref{Pnx-recu} that
\begin{equation}\label{exp:01}
2^{n}\widehat{A}_{n+1}(x)=(1-x)^nP_n'\left(\frac{1+x}{1-x}\right)=\sum_{k=0}^{\lfloor(n+1)/2\rfloor}(n-2k+1)p(n,n-2k+1)(1-x)^{2k}(1+x)^{n-2k}.
\end{equation}

Denote
by $E(n,k,s)$ the coefficients $x^s$ of $(1-x)^{2k}(1+x)^{n-2k}$. Clearly,
$$E(n,k,s)=\sum_{j=0}^{\min (\lrf{\frac{k}{2}},s)}(-1)^j\binom{2k}{j}\binom{n-2k}{s-j}.$$
Then we get the following result.
\begin{cor}
For $n\geq 2$ and $1\leq s\leq n$, we have
\begin{equation*}
\widehat{A}(n+1,s)=\frac{1}{2^{n}}\sum_{k=0}^{\lfloor(n+1)/2\rfloor}(n-2k+1)p(n,n-2k+1)E(n,k,s).
\end{equation*}
\end{cor}

Let $\pi=\pi(1)\pi(2)\cdots \pi(n)\in\msn$.
An {\it interior peak} in $\pi$ is an index $i\in\{2,3,\ldots,n-1\}$ such that $\pi(i-1)<\pi(i)>\pi(i+1)$.
Let $\pk(\pi)$ denote the number of interior peaks of $\pi$. Let $W_n(x)=\sum_{\pi\in\msn}x^{\pk(\pi)}$.
It is well known that the polynomials $W_n(x)$ satisfy the recurrence relation
\begin{equation*}
W_{n+1}(x)=(nx-x+2)W_n(x)+2x(1-x)W'_n(x),
\end{equation*}
with initial conditions $W_1(x)=1,W_2(x)=2$ and $W_3(x)=4+2x$. 
By the theory of {\it enriched P-partitions}, Stembridge~\cite[Remark 4.8]{Stembridge97} showed that
\begin{equation}\label{Stembridge}
W_n\left(\frac{4x}{(1+x)^2}\right)=\frac{2^{n-1}}{(1+x)^{n-1}}A_n(x).
\end{equation}

From~\cite[Theorem 2]{Ma121}, we have
\begin{equation}\label{Wnx-Pnx}
P_n(x)=x^{n-1}(1+x^2)W_n(1+x^{-2}).
\end{equation}
Therefore, combining~\eqref{AnxPnx:01} and~\eqref{Wnx-Pnx}, we get the counterpart of~\eqref{Stembridge}:
\begin{equation*}\label{Anx-Wnx}
W_n\left(\frac{2+2x^2}{(1+x)^2}\right)=\frac{2^{n-1}}{(1+x)^{n-1}}\widehat{A}_n(x).
\end{equation*}
%In other word,
%\begin{equation*}
%\widehat{A}_n(x)=\sum_{k=0}^{\lrf{\frac{n-1}{2}}}\frac{1}{2^{n-k-1}}W(n,k)(1+x^2)^k(1+x)^{n-2k-1}.
%\end{equation*}

\section{Zeros of the alternating Eulerian polynomials}\label{sec:Zeros}
%%%%%%%%%%%%%%%%%%%%%%%%%%%%%%%%%%%%%%%%%%%%%%%%%%%%%%%%%%%%%%%%%%%%%%%%%%%%%%%%%
%%%%%%%%%%%%%%%%%%%%%%%%%%%%%%%%%%%%%%%%%%%%%%%%%%%%%%%%%%%%%%%%%%%%%%%%%%%%%%%%%
%%%%%%%%%%%%%%%%%%%%%%%%%%%%%%%%%%%%%%%%%%%%%%%%%
%%%%%%%%%%%%%%%%%%%%%%%%%%%5%%%%
Combining~\eqref{Pnx-recu} and~\eqref{AnxPnx:01}, it is easy to derive
that
\begin{equation*}\label{recuAnx}
2\widehat{A}_{n+1}(x)=(1+n+2x+nx^2-x^2)\widehat{A}_{n}(x)+(1-x)(1+x^2)\widehat{A}_{n}'(x)
\end{equation*}
for $n\ge 1$.
The bijection $\pi\mapsto \pi^c$ on $\msn$ defined by $\pi^c(i)=n+1-\pi(i)$ shows that $\widehat{A}_{n}(x)$ is symmetric. Hence $(x+1)\|\widehat{A}_{2n}(x)$ for $n\geq 1$.
It is well known that the
classical Eulerian polynomials $A_n(x)$ have only real zeros, and the zeros of $A_n(x)$ separates that of $A_{n+1}(x)$ (see B\'ona~\cite[p.~24]{Bona12} for instance).
Now we present the corresponding result for $\widehat{A}_{n}(x)$.
\begin{theorem}\label{thm:02}
For $n\geq 1$, the zeros of $\widehat{A}_{2n+1}(x)$ and $\widehat{A}_{2n+2}(x)/(1+x)$ are imaginary with multiplicity 1, and
the moduli of all zeros of $\widehat{A}_{n}(x)$ are equal to 1. Furthermore,
the sequence of real parts of the zeros of $\widehat{A}_{n}(x)$ separates that of $\widehat{A}_{n+1}(x)$. More precisely, suppose that $\{r_j\pm \ell_j\mathrm i\}_{j=1}^{n-1}$ are all zeros of $\widehat{A}_{2n}(x)/(1+x)$, $\{s_j\pm t_j\mathrm i\}_{j=1}^n$ are all zeros of $\widehat{A}_{2n+1}(x)$ and $\{p_j\pm q_j\mathrm i\}_{j=1}^{n}$ are all zeros of $\widehat{A}_{2n+2}(x)/(1+x)$, where $-1<r_1<r_2<\cdots<r_{n-1}<0$,
$-1<s_1<s_2<\cdots<s_n<0$ and $-1<p_1<p_2<\cdots<p_{n}<0$.
Then we have
\begin{equation}\label{si01}
 -1<s_1<r_1<s_2<r_2<\cdots<r_{n-1}<s_n,
 \end{equation}
 \begin{equation}\label{si02}
 -1<s_1<p_1<s_2<p_2<\cdots<s_n<p_n.
\end{equation}
\end{theorem}
\begin{proof}
Define $\widetilde{P}_n(x)=\mathrm i^{n-1}P_n(\mathrm i x)$. Then
\begin{equation*}\label{pnx-recu}
\widetilde{P}_{n+1}(x)=(1-x^2)\widetilde{P}'_n(x).
\end{equation*}
From~\cite[Theorem 2]{Ma08}, we get that the polynomials $\widetilde{P}_{n}(x)$ have only real zeros, belong to $[-1,1]$ and the sequence of zeros of $\widetilde{P}_{n}(x)$ separates that of $\widetilde{P}_{n+1}(x)$.
From~\cite[Corollary 8.7]{Hetyei08}, we see that the zeros of the derivative polynomials $P_n(x)$ are pure imaginary with multiplicity 1, belong to
the line segment $[-\mathrm i,\mathrm i]$. In particular, $(1+x^2)\|P_n(x)$. 
Therefore, the polynomials $P_{2n+1}(x)$ and $P_{2n+2}(x)$ have the following expressions: $$P_{2n+1}(x)=(1+x^2)\prod_{i=1}^n (x^2+a_i), P_{2n+2}(x)=x(1+x^2)\prod_{i=1}^n (x^2+b_i),$$
where
\begin{equation}\label{aibi}
0<a_1<b_1<a_2<b_2<\cdots<a_n<b_n<1.
\end{equation}
Using~\eqref{AnxPnx:01}, we get
\begin{align*}
2^{2n}\widehat{A}_{2n+1}(x)&=\prod_{i=1}^n((1+x)^2+a_i(1-x)^2),\\
2^{2n+1}\widehat{A}_{2n+2}(x)&=(1+x)\prod_{i=1}^n((1+x)^2+b_i(1-x)^2).
\end{align*}
Hence
\begin{align*}
2^{2n}\widehat{A}_{2n+1}(x)&=\prod_{i=1}^n(1+a_i)\left(x+\frac{1-a_i}{1+a_i}+ \frac{2\mathrm i\sqrt{a_i}}{1+a_i}\right)\left(x+\frac{1-a_i}{1+a_i}-\frac{2\mathrm i\sqrt{a_i}}{1+a_i}\right),\\
2^{2n+1}\widehat{A}_{2n+2}(x)&=(1+x)\prod_{i=1}^n(1+b_i)\left(x+\frac{1-b_i}{1+b_i}+ \frac{2\mathrm i\sqrt{b_i}}{1+b_i}\right)\left(x+\frac{1-b_i}{1+b_i}-\frac{2\mathrm i\sqrt{b_i}}{1+b_i}\right).
\end{align*}
Since
$$\left(\frac{1-a_i}{1+a_i}\right)^2+\left(\frac{2\sqrt{a_i}}{1+a_i}\right)^2=\left(\frac{1-b_i}{1+b_i}\right)^2+\left(\frac{2\sqrt{b_i}}{1+b_i}\right)^2=1,$$
the moduli of all zeros of $\widehat{A}_{n}(x)$ are equal to 1.
Note that $$s_j=-\frac{1-a_j}{1+a_j},~p_j=-\frac{1-b_j}{1+b_j}.$$
So~\eqref{si02} is immediate. Along the same lines, one can get~\eqref{si01}.
\end{proof}

From Theorem~\ref{thm:02}, we immediately get that the sequence of imaginary parts of the zeros of $\widehat{A}_{n}(x)$ also separates that of $\widehat{A}_{n+1}(x)$.


\begin{thebibliography}{9}

\bibitem{Bona12}
M. B\'ona, Combinatorics of Permutations, second edition, CRC Press, Boca Raton, FL, 2012.

\bibitem{Chebikin08}
D. Chebikin, Variations on descents and inversions in permutations, \newblock {\em Electron. J.
Combin.} {15} (2008), \#R132.

\bibitem{Hetyei08}
G. Hetyei, Tchebyshev triangulations of stable simplicial complexes,
\newblock {\em J. Combin. Theory Ser. A} 115 (2008), 569--592.


\bibitem{Hoffman95}
M.E. Hoffman, Derivative polynomials for tangent and secant, \newblock {\em Amer. Math. Monthly} 102 (1995) 23--30.

\bibitem{Gessel14}
I.M. Gessel, Y. Zhuang, Counting Permutations by alternating descents,  \newblock {\em Electron. J.
Combin.} {21(4)} (2014), \#P4.23.

\bibitem{Ma121}
S.-M. Ma, Derivative polynomials and enumeration of permutations by number of interior and left peaks, \newblock {\em Discrete Math.} 312 (2012), 405--412.

\bibitem{Ma122}
S.-M. Ma, An explicit formula for the number of permutations with a given number of alternating runs, \newblock {\em J. Combin. Theory Ser. A} 119 (2012), 1660--1664.

\bibitem{Ma08}
S.-M. Ma, Y. Wang, $q$-Eulerian polynomials and polynomials with only real zeros, \newblock {\em Electron. J. Combin.} 15 (2008), \#R17.

\bibitem{Remmel12}
J.B. Remmel, Generating Functions for Alternating Descents and
Alternating Major Index,  \newblock {\em Ann. Comb.,} {16} (2012), 625--650.


\bibitem{Sloane}
N.J.A. Sloane, \newblock {\em The On-Line Encyclopedia of Integer Sequences},
http://oeis.org.


\bibitem{Stembridge97}
J. Stembridge, \emph{Enriched P-partitions},  \newblock {\em Trans. Amer. Math. Soc.} 349(2): (1997) 763--788.
\end{thebibliography}
\end{document}